\documentclass[12pt]{article}
\usepackage{amssymb,amsmath,latexsym}
\usepackage{amsthm}
\numberwithin{equation}{section}
\newtheorem{theorem}{Theorem}

\begin{document}
\author{Alexander E Patkowski}
\title{On some new Bailey pairs and new expansions for some Mock theta functions}
\date{\vspace{-5ex}}
\maketitle
\abstract{In this paper we offer some new identities associated with mock theta functions and establish new Bailey pairs related to indefinite quadratic forms. We believe our proof is instructive use of changing base of Bailey pairs, and offers new information on some Bailey pairs that have proven important in the study of Mock theta functions.}

\section{Introduction}
In the last few decades, considerable research has been done on establishing special Bailey pairs to relate $q$-series to indefinite binary quadratic forms. (For a definition of Bailey's lemma and its history see Warnaar [14].) This dates back to the work of Andrews [1], where in his study of mock theta functions, he found that
\begin{equation}(q)_{\infty}\sum_{n\ge0}\frac{q^{n^2}}{(1+q)(1+q^2)\cdots(1+q^n)}=\sum_{n\ge0}q^{n(5n+1)/2}(1-q^{4n+2})\sum_{|j|\le n}(-1)^jq^{-j^2},\end{equation}
which is Ramanujan's fifth order mock theta function $f_0(q).$ Subsequent work includes Zwegers thesis [14] which established a connection between the right side of (1.1) and real analytic modular forms through rewriting the summation bounds. For some more related material on mock theta functions see [7, 9, 10, 12]. \par The purpose of this paper is to collect some observations from previous work [1, 2] to construct new identities for mock theta functions by establishing
some interesting Bailey pairs. A subsequent consequence of this is a variety of new identities for $q$-series related to indefinite binary quadratic forms. 
\section{The Bailey pairs} 
In this section we will prove the desired Bailey pairs to obtain new relations among $q$-series related to indefinite quadratic forms in the following sections. To do this
we will need to state some known results and apply the same technique used in [5]. We use the notation $(\alpha_n(a,q), \beta_n(a,q))$ to be a Bailey pair where
\begin{equation}\beta_n(a,q^{k})=\sum_{0\le j\le n}\frac{\alpha_n(a,q^{k})}{(q^{k};q^{k})_{n-j}(aq^{k};q^{k})_{n+j}}.\end{equation}
If $k=1,$ then for convenience we simply put $(\alpha_n(a,q), \beta_n(a,q))=(\alpha_n, \beta_n).$ 
\\*
{\bf Lemma 2.1. [2]} \it The pair $(\alpha_n^{*}, \beta_n^{*})$ is a Bailey pair relative to $a$ where

\begin{equation}\alpha_n^{*}(a,b,c,q)=\frac{q^{n^2}(bc)^n(1-aq^{2n})(a/b)_n(a/c)_n}{(1-a)(bq,cq)_n}\sum_{0\le j\le n}\frac{(-1)^j(1-aq^{2j-1})(a)_{j-1}(b,c)_j}{q^{j(j-1)/2}(bc)^j(q,a/b,a/c)_j},\end{equation}
\begin{equation}\beta_n^{*}(a,b,c,q)=\frac{1}{(bq,cq)_n}.\end{equation}
\rm
We will also need some lemmas that follow from well-known special cases of Bailey's lemma.
\\*
{\bf Lemma 2.2. [5, (S2)]} \it If $(\alpha_n, \beta_n)$ is a Bailey pair relative to $a$ then so is $(\alpha_n', \beta_n')$ where

\begin{equation}\alpha_n'=a^{n/2}q^{n^2/2}\alpha_n,\end{equation}
\begin{equation}\beta_n'=\frac{1}{(-\sqrt{aq})_n}\sum_{0\le j\le n}\frac{(-\sqrt{aq})_ja^{j/2}q^{j^2/2}}{(q)_{n-j}}\beta_j.\end{equation}
{\bf Lemma 2.3. [5, (S1)]} If $(\alpha_n, \beta_n)$ is a Bailey pair relative to $a$ then so is $(\alpha_n', \beta_n')$ where

\begin{equation}\alpha_n'=a^{n}q^{n^2}\alpha_n,\end{equation}
\begin{equation}\beta_n'=\sum_{0\le j\le n}\frac{a^{j}q^{j^2}}{(q)_{n-j}}\beta_j.\end{equation}
\rm
Next is the reverse implication of [5, (D1)].
\\*
{\bf Lemma 2.4.} \it If $(\alpha_n, \beta_n)$ is a Bailey pair relative to $a$ then so is $(\alpha_n', \beta_n')$ where

\begin{equation}\alpha_n'(a^2,q^2)=\alpha_n,\end{equation}
\begin{equation}\beta_n'(a^2,q^2)=\frac{1}{(-aq)_{2n}}\sum_{0\le j\le n}\frac{(-1)^{n-j}q^{(n-j)^2}}{(q^2;q^2)_{n-j}}\beta_j.\end{equation}
\rm
Now we offer some new Bailey pairs.
\\*
{\bf Lemma 2.5.} \it The pair $(\alpha_n, \beta_n)$ is a Bailey pair relative to $q$ where
\begin{equation}\beta_n=\frac{(-1)^n(1-c)}{(q^2;q^2)_n(1-cq^n)},\end{equation}
\begin{equation}\alpha_n=q^{-n(n+1)/2}\alpha_n^{*}(q,-1,c,q),\end{equation}
where the $\alpha_n^{*}$ is given in (2.2).
\rm
\begin{proof} Put $a=q$ in Lemma 2.2, and suppose the left sides of (2.4)--(2.5) are given by the $a=q,$ $b=-1,$ case of (2.2)--(2.3). Then
we would need our $\alpha_n$ and $\beta_n$ to be given as in Lemma 2.5 by the uniqueness of Bailey pairs, because
\begin{equation}\frac{1}{(-q)_n(cq)_n}=\frac{1}{(-q)_n}\sum_{0\le j\le n}\frac{(-1)^jq^{j(j+1)/2}(1-c)}{(q)_{n-j}(q)_j(1-cq^j)}.\end{equation}\end{proof}
We need one more lemma before our main desired Bailey pair.
\\*
{\bf Lemma 2.6.} \it The pair $(\alpha_n, \beta_n)$ is a Bailey pair relative to $q$ where
\begin{equation}\beta_n=2\sum_{0\le j \le n}\frac{(-1)^jq^{j(j+1)}}{(q)_{n-j}(q^2;q^2)_j(1-c^2q^{2j})},\end{equation}
\begin{equation}\alpha_n=q^{n(n+1)/2}\left(\frac{\alpha_n^{*}(q,-1,c,q)}{(1-c)}+\frac{\alpha_n^{*}(q,-1,-c,q)}{(1+c)}\right),\end{equation}
where the $\alpha_n^{*}$ is given in (2.2). \rm
\begin{proof} To prove this pair we take the pair in Lemma 2.5, multiply both sides by $(1-c)^{-1},$ and then add the same pair to itself after $c$ has been replaced by $-c,$ to get

\begin{equation}\beta_n=2\frac{(-1)^n}{(q^2;q^2)_n(1-c^2q^{2n})},\end{equation}
\begin{equation}\alpha_n=q^{-n(n+1)/2}\left(\frac{\alpha_n^{*}(q,-1,c,q)}{(1-c)}+\frac{\alpha_n^{*}(q,-1,-c,q)}{(1+c)}\right).\end{equation}
Now insert this Bailey pair into Lemma 2.3. This follows from the observing that
\begin{equation}\frac{1}{1-cq^j}+\frac{1}{1+cq^j}=\frac{2}{1-c^2q^{2j}}.
\end{equation}
\end{proof}
We now obtain our main result concerning Bailey pairs.\newline
{\bf Lemma 2.7.} \it The pair $(\alpha_n, \beta_n)$ is a Bailey pair relative to $q^2$ where
\begin{equation}\beta_n(q^2,q^2)=\frac{2}{(-q^2)_{2n}(c^2;q^2)_{n+1}},\end{equation}
\begin{equation}\alpha_n(q^2,q^2)=q^{n(n+1)/2}\left(\frac{\alpha_n^{*}(q,-1,c,q)}{(1-c)}+\frac{\alpha_n^{*}(q,-1,-c,q)}{(1+c)}\right),\end{equation}
where the $\alpha_n^{*}$ is given in (2.2). \rm
\begin{proof} To obtain this pair, insert the pair in Lemma 2.6 into Lemma 2.4. This takes a Bailey pair relative to $(q,q)$ and produces a Bailey pair relative to $(q^2,q^2)$. This step is attractive because it will convert a pair with a $\beta_n$ that involves a sum into a simple truncated product. Note that the $\alpha_n$ computation is trivial, and for
the $\beta_n$ we compute
$$ \beta_n(q^2,q^2)=\frac{1}{(-q^2)_{2n}}\sum_{0\le j\le n}\frac{(-1)^{n-j}q^{(n-j)^2}}{(q^2;q^2)_{n-j}}\beta_j$$
\begin{equation}=\frac{2}{(-q^2)_{2n}}\sum_{0\le j\le n}\frac{(-1)^{n-j}q^{(n-j)^2}}{(q^2;q^2)_{n-j}}\left(\sum_{0\le i \le j}\frac{(-1)^iq^{i(i+1)}}{(q)_{j-i}(q^2;q^2)_i(1-c^2q^{2i})}\right).\end{equation}
Notice that the sum in (2.20) may be viewed as the coefficient of $z^n$ in
\begin{equation}\frac{(zq;q^2)_{\infty}}{(z)_{\infty}}\sum_{i\ge0}\frac{(-z)^iq^{i(i+1)}}{(q^2;q^2)_i(1-c^2q^{2i})},  \end{equation}
which is the coefficient of $z^n$ in
\begin{equation}\frac{1}{(z;q^2)_{\infty}}\sum_{i\ge0}\frac{(-z)^iq^{i(i+1)}}{(q^2;q^2)_i(1-c^2q^{2i})}.  \end{equation}
On the other hand, taking coefficients of $z^n$ in (2.22), we can see this is just
\begin{equation} \frac{1}{(c^2;q^2)_{n+1}}, \end{equation}
by equation (2.12). Hence the pair follows.
\end{proof}
It should immediately be observed by the reader that the $c\rightarrow0$ case is a well-known Bailey pair from the study [2], used to establish sixth order mock theta function expansions. It is therefore
 natural that the case $c=q^{1/2}$ would reveal new information.

\section{A related Bailey pair} 

This section notes the Bailey pair obtained by considering the taking of the pair in Lemma 2.5. and instead subtracting the
pair from itself when $c$ has been replaced by $-c.$ This is 

\begin{equation}\beta_n=2\frac{(-1)^ncq^n}{(q^2;q^2)_n(1-c^2q^{2n})},\end{equation}
\begin{equation}\alpha_n=q^{-n(n+1)/2}\left(\frac{\alpha_n^{*}(q,-1,c,q)}{(1-c)}-\frac{\alpha_n^{*}(q,-1,-c,q)}{(1+c)}\right).\end{equation}
Recall from [5, (L1)] that if $(\alpha_n, \beta_n)$ is a Bailey pair relative to $q$ then $(\alpha_n', \beta_n')$ is a Bailey pair relative to $1$ where
\begin{equation} \beta_n'=\sum_{j\ge0}^{n}\frac{q^{j^2}}{(q)_{n-j}}\beta_j\end{equation}
\begin{equation} \alpha_n'=(1-q)q^{n^2}\left(\frac{\alpha_n}{1-q^{2n+1}}-q^{2n-1}\frac{\alpha_{n-1}}{1-q^{2n-1}}\right).\end{equation}
Inserting (3.1)--(3.2) into this result gives us the Bailey pair $(\alpha_n^{1}, \beta_n^{1})$ relative to $a=1$ where
\begin{equation} \beta_n^{1}=2c\sum_{j\ge0}^{n}\frac{(-1)^jq^{j^2+j}}{(q^2;q^2)_j(q)_{n-j}(1-c^2q^{2j})}\end{equation}
\begin{equation} \alpha_n^{1}=(1-q)q^{n^2}\left(\frac{a_n}{1-q^{2n+1}}-q^{2n-1}\frac{a_{n-1}}{1-q^{2n-1}}\right),\end{equation}
where the $a_n$ is given by (3.2). We leave the $\alpha_n^{1}$ is this form for brevity. Applying the $a=1$ case of Lemma 2.4 gives us a new pair.
\\*
\\*
{\bf Lemma 3.1.} \it We have, the Bailey pair $(\alpha_n^{1'}, \beta_n^{1'})$ relative to $a=1$ where 
\begin{equation}\beta_n^{1'}(1,q^2)=\frac{2c}{(-q)_{2n}(c^2;q^2)_{n+1}},\end{equation}
\begin{equation}\alpha_n^{1'}(1,q^2)=(1-q)q^{n^2}\left(\frac{a_n}{1-q^{2n+1}}-q^{2n-1}\frac{a_{n-1}}{1-q^{2n-1}}\right),\end{equation}
where the $a_n$ is given in (3.2). \rm

\section{Mock theta functions} 

This section presents our main results of the paper, establishing new identities for mock theta functions, and introducing new $q$-series. 
\par
It is a clear observation that the case $c=q^{1/2}$ of Lemma 2.7 (and replacing $q$ by $q^2$) is the Bailey pair $(\alpha_n, \beta_n)$ relative to $(q^4, q^4)$ where
\begin{equation}\beta_n(q^4,q^4)=\frac{2}{(-q^4;q^2)_{2n}(q^2;q^4)_{n+1}},\end{equation}
\begin{equation}\alpha_n(q^4,q^4)=q^{n(n+1)}\left(\frac{(-1)^nq^{2n^2+n}(1+q^{2n+1})}{(1-q^2)}\sum_{|j|\le n}q^{-j^2}+\frac{q^{2n^2+n}(1-q^{2n+1})}{(1-q^2)}\sum_{|j|\le n}(-1)^jq^{-j^2}\right),\end{equation} We now state some interesting consequences.
\begin{theorem} We have,
\begin{equation}2\sum_{n\ge0}\frac{q^{2n(n+1)}}{(q^4;q^8)_{n+1}}=\frac{(-q^4;q^4)_{\infty}}{(-q^2;q^2)_{\infty}}\left(P_1(q)+P_1(-q)\right),\end{equation}
where $P_1(q)=\sum_{n\ge0}\frac{q^{2n^2+2n}}{(-q^2;q^2)_n}\sum_{0\le j\le n}\frac{q^{j^2+j}}{(-q;q^2)_{j+1}(q^2;q^2)_{n-j}}.$
\end{theorem}
\begin{proof} In Bailey's lemma, we take the $\rho_1\rightarrow\infty,$ $\rho_2=-q,$ case to get
\begin{equation}2\sum_{n\ge0}\frac{(-q^4;q^4)_nq^{2n(n+1)}}{(-q^2;q^2)_{2n+1}(q^2;q^4)_{n+1}}=\frac{(-q^4;q^4)_{\infty}}{(q^4;q^4)_{\infty}}\left(G_1(q)+G_1(-q)\right),\end{equation}
where $G_1(q)=\sum_{n\ge0}q^{5n^2+4n}(1-q^{2n+1})\sum_{|j|\le n}(-1)^jq^{-j^2}.$ Now using (through a simple application of Lemma 2.2) the equation \begin{equation}\sum_{n\ge0}\frac{q^{2n^2+2n}}{(-q^2;q^2)_n}\sum_{0\le j\le n}\frac{q^{j^2+j}}{(-q;q^2)_{j+1}(q^2;q^2)_{n-j}}=\frac{1}{(q^2;q^2)_{\infty}}\sum_{n\ge0}q^{5n^2+4n}(1-q^{2n+1})\sum_{|j|\le n}(-1)^jq^{-j^2},\end{equation} we complete the proof.
\end{proof}
Note that the $q$-series on the left side of (4.3) is one of Ramanujan's 10th order mock theta functions $\phi(q^4)$ [6]. The function $P_1(q)$ may take on many different multiple sum forms, here we only mention the form obtained from using Lemma 2.2. \par
Naturally we may offer more new examples of the mock theta function type like Theorem 1.

\begin{theorem} We have,

\begin{equation}2\sum_{n\ge0}\frac{q^{4n(n+1)}}{(-q^4;q^4)_n(q^4;q^8)_{n+1}}=\frac{1}{(q^4;q^4)_{\infty}}\left(G_2(q)+G_2(-q)\right),\end{equation}
where $G_2(q)=\sum_{n\ge0}q^{7n^2+6n}(1-q^{2n+1})\sum_{|j|\le n}(-1)^jq^{-j^2}.$
\end{theorem}

The $q$-series on the left side of (4.6) is one of Ramanujan's 7th order mock theta functions $F_2(q^4).$ \par Zwegers [16, Theorem 1] offered an expansion for the fifth order mock theta function
$\chi_1(q)$ (and its companion $\chi_0$) as a triple sum. It may equivalently be viewed as a special instance of a known Bailey pair with the conjugate Bailey pair due to Bessoud and Singh [4, 13],
\begin{equation}\delta_n=q^n,\end{equation}
\begin{equation}\gamma_n=\frac{q^n}{(q)_{\infty}^2}\sum_{j\ge0}(-1)^jq^{j(j+1)/2+2nj}.\end{equation}

\begin{theorem} We have,
\begin{equation}2\chi_1(q^4)=2\sum_{n\ge0}\frac{q^{4n}}{(q^{4n+4};q^4)_{n+1}}=\frac{1}{(q^4;q^4)_{\infty}^2}\left(G_3(q)+G_3(-q)\right),\end{equation}
where $G_3(q)=\sum_{n,i\ge0}q^{3n^2+6n+2i(i+3)+8ni}(1-q^{2n+1})\sum_{|j|\le n}(-1)^jq^{-j^2}.$
\end{theorem}

In the previous three theorems we can conclude that, in a sense, the ``even part" of $G_i(q)$ for $i=1,2,3$ is a mixed mock modular form [11]. However, more information on these $G_i(q)$ is needed, which will be dealt with in the next section. \par If we appeal to a certain transformation formula found in Fine's text [7, pg.18, eq.(16.3)], and [7, pg.6, eq.(7.31)], we may utilize the pair (2.15)--(2.16) with the conjugate pair (4.7)--(4.8) to get the following result.

\begin{theorem} We have,
\begin{equation}2\sum_{n\ge0}\frac{(-t^2)^nq^{n^2}}{(t^2;q^2)_{n+1}}=\frac{(q;q^2)_{\infty}}{(q^2;q^2)_{\infty}^2}\sum_{n,j\ge0}(-1)^jq^{j(j+1)/2+2nj-n(n-1)/2}\left(\frac{\alpha_n^{*}(q,-1,t,q)}{(1-t)}+\frac{\alpha_n^{*}(q,-1,-t,q)}{(1+t)}\right). \end{equation}

\end{theorem}
This expansion has relevance to mock theta functions in special cases, such as $t=i.$
\newline

We mention there are also certain product expansions that are natural consequences of these types of pairs. One example from (2.15)--(2.16) is:
\begin{equation} \frac{(q^2;q^2)_{\infty}^2}{(t^2;q^2)_{\infty}}=\sum_{n\ge0}q^{n(n+1)/2}\left(\frac{\alpha_n^{*}(q,-1,t,q)}{(1-t)}+\frac{\alpha_n^{*}(q,-1,-t,q)}{(1+t)}\right).\end{equation}

\section{More on $G_i(q)$ in relation to Mock theta functions} 

Recent work from [10] gives us some tools to state a result on our $G_i(q).$ We can certainly see that $G_i(q)$ for $i=1,2,3$ cannot be a modular form. We recall the function 
\begin{equation} f_{a,b,c}(x,y,q):=\left(\sum_{r,s\ge0}-\sum_{r,s<0}\right)(-1)^{r+s}x^{r}y^{s}q^{ar(r-1)/2+brs+cs(s-1)/2}.\end{equation}
The Appell-Lerch series is of the form
\begin{equation}m(x,q,z):=\frac{1}{(x)_{\infty}(q/x)_{\infty}(q)_{\infty}}\sum_{j\in\mathbb{Z}}\frac{(-1)^jq^{j(j-1)/2}z^j}{1-q^{j-1}xz}.\end{equation}
Recent studies [10, 15] have shown specializations of (5.1) can be expressed in terms of (5.2), which allows us to deduce which functions are mock theta functions (see [15 , Chapter 1]). 
\begin{theorem} We have
\begin{equation} \frac{(-q^4;q^4)_{\infty}}{(q^4;q^4)_{\infty}}G_1(q),\end{equation}
\begin{equation} \frac{1}{(q^4;q^4)_{\infty}}G_2(q),\end{equation}
and
\begin{equation} \frac{1}{(q^4;q^4)_{\infty}^2}G_3(q),\end{equation} are mock theta functions.
\end{theorem}
Note that this theorem implies $P_1(q)$ and similar multi-sums related to $G_i(q)$ are mixed mock modular forms.
\begin{proof} We give the details for $G_1(q)$ and $G_2(q)$ only. We write 
$$G_1(q)=\sum_{n\ge0}\sum_{|j|\le n}(-1)^jq^{5n^2+4n-j^2}-\sum_{n< 0}\sum_{|j|\le -n-1}(-1)^jq^{5n^2+4n-j^2}$$
$$=\left(\sum_{r,s\ge0}-\sum_{r,s<0}\right)(-1)^{r+s}q^{4r^2+12rs+4s^2+4r+4s}+\left(\sum_{r,s\ge0}-\sum_{r,s<0}\right)(-1)^{r+s}q^{4r^2+12rs+4s^2+14r+14s+9}$$
$$=f_{8,12,8}(q^8,q^8,q)+q^9f_{8,12,8}(q^{18},q^{18},q).$$ This is mixed-mock.
For $G_2(q)$ we have
$$G_2(q)=\sum_{n\ge0}\sum_{|j|\le n}(-1)^jq^{7n^2+6n-j^2}-\sum_{n< 0}\sum_{|j|\le -n-1}(-1)^jq^{7n^2+6n-j^2}$$
$$=\left(\sum_{r,s\ge0}-\sum_{r,s<0}\right)(-1)^{r+s}q^{6r^2+16rs+6s^2+6r+6s}+\left(\sum_{r,s\ge0}-\sum_{r,s<0}\right)(-1)^{r+s}q^{6r^2+16rs+6s^2+20r+20s+13}$$
$$=f_{12,16,12}(q^{12},q^{12},q)+q^{13}f_{12,16,12}(q^{26},q^{26},q).$$ This is mixed-mock. The above is done by replacing the second series over $n$ with $-n-1$ and then making the substitutions $n=r+s$ and $j=r-s.$ Note also $f_{8,12,8}(q^8,q^8,q)=f_{2,3,2}(q^8,q^8,q^4).$ Therefore we need to consider $(-q)_{\infty}f_{2,3,2}(q^2,q^2,q)/(q)_{\infty},$ which is included in the Appell-Lerch sum representations in [10, Theorem 0.6]. That is, $f_{2,3,2}(q^8,q^8,q^4)=j(q^4,q^8)\phi(q^4)$ where $j(q^a,q^m)=\sum_{n\in\mathbb{Z}}(-1)^nq^{mn(n-1)/2+an},$ by [10, eq.(9.8)]. Similarly we have by [10, eq.(9.15)]  $f_{12,16,12}(q^{12},q^{12},q) = f_{3,4,3}(q^{12},q^{12},q^4) = j(q^4,q^{12}) F_2(q^4).$ \end{proof}

{\bf Acknowledgement.} We thank Eric Mortenson for some helpful comments on his paper [10], and noting [10, eq.(9.15)], and [10, eq.(9.8)]. We also thank the referee for corrections and improving the details in the proofs in the paper.

1390 Bumps River Rd. \\*
Centerville, MA
02632 \\*
USA \\*
E-mail: alexpatk@hotmail.com
\end{document}